\documentclass[11pt,twoside,bezier]{article}
\usepackage{amssymb,amsthm}
\newcommand{\V}{\mathfrak{V}}
\newcommand{\U}{\mathfrak{U}}

\newcommand{\la}{\langle}
\newcommand{\ra}{\rangle}

\newcounter{ppp}
\theoremstyle{plain}\newtheorem{Thm}{Theorem}
\newtheorem{Lemma}{Lemma}

\usepackage{imakeidx}
\makeindex[name=g,title=Subject index]
\usepackage[columns=2]{idxlayout}

\begin{document}

\theoremstyle{plain}
\newtheorem{theo}{Theorem}[section]
\newtheorem{lm}[theo]{Lemma}
\newtheorem{cy}[theo]{Corollary}
\newtheorem{df}[theo]{Definition}
\newtheorem{remark}[theo]{Remark}
\newtheorem{prop}[theo]{Proposition}
\newtheorem{prob}[theo]{Problem}

\title{Finite groups and rings generating varieties of rapid growth}
\author{Alexander Olshanskii}
\date{}
\maketitle
\begin{abstract}
Let $A$ be a finite universal algebra.  Then the cardinalities of the $n$-generated free algebras $F_n$ in the variety (equational class) generated by $A$ satisfy G. Birkhoff's inequality: $|F_n|\le |A|^{|A|^n}$ for $n=1,2,\dots$ It follows that $\limsup_{n\to\infty}\sqrt[n]{\log |F_n|}\le |A|$. When $A$ is a finite group or a finite nonassociative 
algebra, we obtain a criterion for equality in this estimate; equivalently, a criterion for maximal growth of the sequence  $\{|F_n|\}_{n=1}^{\infty}$. 
\end{abstract}

\medskip


\medskip


\section{Introduction}

Let $A$ be a finite algebra of a fixed signature. The variety of algebras $var\,A$ generated by $A$ consists of all algebras of the same signature satisfying every identity that holds in $A$. 
A basic theorem of G. Birkhoff \cite{B, HM} says that $var\;A$ is the smallest class of algebras
containing $A$ and closed under taking subalgebras, homomorphic images,
and Cartesian products.

The cardinality of every $n$-generated algebra in the variety $var\,A$ is bounded by
the cardinality of the $n$-generated free algebra $F_n$ in $var\;A$.  The algebra $F_n$ isomorphically
embeds in a direct power of the algebra $A$, and its cardinality
is bounded from above by the double-exponential function $|A|^{|A|^n}$, where
$|A|$ is the cardinality of the algebra $A$ (see \cite{B} or subsection 0.3 in \cite{HM}). It follows that $\limsup_{n\to\infty}\sqrt[n]{\log |F_n|}\le |A|$. We denote the left-hand side of this inequality by $pot\;A$ and call
it the {\it potential} of the finite algebra $A$. 

For example, if $G$ is a 
nontrivial finite group, then $pot\;G =1$ if $G$ is nilpotent
and $pot\;G \ge 2$ otherwise (\cite{HN}, 24.52). The same alternative holds for
finite-dimensional (non-associative) algebras over finite fields \cite{BO}.

As early as 1937, it was noticed by B.H. Neumann \cite{BN} that  Birkhoff's estimate can be improved in many cases. Our goal is to find necessary and sufficient conditions on $A$ under which the potential $pot\;A$ attains its maximal possible value, namely conditions equivalent to the equality  $pot\;A = |A|$, where $A$ is either a finite group or a finite, not necessarily associative, ring.

Recall that a nontrivial group (ring) $A$ is called {\it monolithic} if the intersection of all nontrivial normal subgroups (2-sided ideals) of $A$ is nontrivial.
This intersection is the {\it monolith} of the monolithic group (respectively, ring) $A$.

\begin{Thm}\label{th1} Let $G$ be a finite group.  
The equality
{\em pot}\;$G=|G|$ holds if and only if 

(*) $G$ is a monolithic group with a non-abelian monolith.

Moreover, under condition (*), there exist positive constants $c_1$ and $c_2$ such that\\
$c_1|G|^n <\log |F_n|< c_2 |G|^n$ for the $n$-generated free groups $F_n$ in the variety
$var\; G$.
\end{Thm}

Many nonsimple groups satisfy condition (*); for example,
the symmetric groups $S_n$ for $n\ge 5$ and any wreath product of a non-cyclic finite
simple group and a nontrivial finite group.

A (not necessarily associative) non-zero ring $A$ is a {\it prime ring} if the
product of any two nonzero (2-sided) ideals of $A$ is nonzero. The following is an
analogue of  Theorem \ref{th1}.

\begin{Thm}\label{th2} Let $A$ be a (non-associative) finite ring. 
The equality {\em pot}\;$A=|A|$ holds if and only if $A$ is a prime ring.

Moreover, if $A$ is a prime ring, then there exist positive constants $c_1$ and $c_2$ such that\\
$c_1|A|^n <\log |F_n|< c_2 |A|^n$ for the $n$-generated free rings $F_n$ in the variety
$var\;A$.
 
\end{Thm}

Classical results in ring theory imply that every finite prime associative
ring is simple  (\cite{VDW}, Section 98). By contrast, the following construction gives
a nonsimple prime algebra over a field with basis $(e,f)$ and multiplication rules
 $e^2 = f$, $ef=f$, $fe=0$, and $f^2=f$. It has only one one-dimensional ideal
$\langle f\rangle$. 

\section{Sufficiency in Theorem 1}\label{s}

\begin{Lemma}\label{lem1} Let $G$ be a finite monolithic group with
a non-abelian monolith. Let $H$ be a direct product $H_1\times\dots\times H_k$ with projections $\pi_i : H \to H_i$. Let $S$ be a subdirect
product, i.e. a subgroup of $H$ such that $\pi_i(S) = H_i$ for all $i$, with an epimorphism $\phi : S \to G$.
Then there exist an index $i$ and an epimorphism $\psi: H_i \to G$ such that $\phi = \psi\pi_i$.
\end{Lemma}
\proof By induction on $k$; this is trivial when $k \le 1$; assume $k \ge 2$. The subgroups $\phi(S\cap H_i)$ are
normal and pairwise centralize each other, so at most one of them is nontrivial since the monolith of $G=\phi(S)$ is non-abelian. After reindexing,
we may assume $\phi(S\cap H_k)$ is trivial. Let $\pi$
be the projection $S \to H_1\times\dots\times H_{k-1}$ and $S' = \pi(S)$.
Then $\phi\pi^{-1}$ is a well-defined epimorphism $S'\to G$ since
$\ker\pi = S\cap H_k$ and $\phi(S\cap H_k)=\{1\}$.
By induction, there exists $i < k$ and an epimorphism $\psi: H_i \to G$ such that $\phi\pi^{-1}= \psi\pi'_i$, where $\pi'_i$ is the projection of
$H_1 \times\dots\times H_{k-1}$ onto $H_i$. Hence $\phi = \psi\pi'_i\pi =\psi\pi_i$, as required.
\endproof

\begin{Lemma}\label{lem2} Let $G$ be a finite monolithic 
group with non-abelian monolith $M$. 
Let $G^I$ be a finite direct power of $G$, with direct
factors $G_i$ isomorphic to $G$ for $i\in I$. For
a fixed $n\ge 1$, let $x_1,\dots,x_n\in G^I$ and 
$x_{1i},\dots x_{ni}$ be their projections in $G_i$. 
Assume that the elements $x_{1i},\dots x_{ni}$ generate $G_i$ for $i\in I$ and
suppose the $n$-tuples $(x_{1i},\dots x_{ni})$, $i\in I$, lie in pairwise distinct
(${\rm Aut}\, G$)-orbits on $G^n$. Then the subgroup $B$ of $G^I$ generated
by $x_1,\dots,x_n$ contains $M^I$.
\end{Lemma}
\proof Arguing by contradiction, assume that $B$ does not contain the copy $M_i$ of the monolith
$M$ for some $i\in I$. Since $M_i$ is the
monolith  of $G_i$, it follows that $B\cap M_i = \{1\}$, because the subgroup $B\cap M_i$
of the subdirect product $B$ must be normal in the factor $G_i$. Therefore
$B\cap G_i = \{1\}$ since nontrivial normal subgroups of $G_i$ contain $M_i$.
Hence $B$ can be subdirectly embedded in the product $\prod_{j\ne i} G_j$. 

By Lemma \ref{lem1}, the projection $\pi_i: B\to G_i$ factors as 
$\psi\pi_j$, where $j\ne i$ and $\psi$ is an isomorphism $G_j\to G_i$
Thus, the $n$-tuple $(x_1,\dots,x_n)$ is mapped 
to $(x_{1i},\dots,x_{ni})$ under $\pi_i$ and it is mapped to $(x_{1j},\dots,x_{nj})$
under $\psi$. It follows that the isomorphism $\psi$ maps $(x_{1j},\dots,x_{nj})$
to $(x_{1i},\dots,x_{ni})$, contrary to the assumption that these $n$-tuples
belong to different ${\rm Aut}\,G$-orbits. The lemma is proved by contradiction.
\endproof

\medskip

{\bf Remark 1.} {\em For finite simple, non-abelian groups $G$, the statement of Lemma \ref{lem2}
has been known for a long time, see \cite{PH}.}

\medskip

For each proper subgroup $H\le G$, we have $|H|\le |G|/2$ and the number of $n$-tuples of elements from $H$ is equal to $|H|^n$. If $s$ is the number of proper subgroups of $G$, then the number of $n$-tuples generating $G$ is at least $|G|^n-s(|G|/2)^n=|G|^n(1-O(2^{-n}))$.
Thus, for some constant $c>0$ and all sufficiently large $n$, there are at least $\frac{|G|^n}{|{\rm Aut}\,G|}\left(1-O(2^{-n})\right)> c|G|^n $ 
$n$-tuples lying in pairwise distinct (${\rm Aut}\; G$)-orbits. By Lemma \ref{lem2},
there is an $n$-generated subgroup $B=B(n)$ of order at least $|M|^{c|G|^n}$ in some $G^I$.
Since $B(n)\in var\; G$, we obtain
$\log|F_n|\ge\log |B(n)|> c_1 |G|^n$ for some $c_1>0$ and every $n\ge 1$, as required. Birkhoff's inequality
 implies $\log|F_n|< c_2 |G|^n$ for some $c_2>0$.  This proves the sufficiency part of Theorem \ref{th1}.

 \section{Necessity in Theorem 1}

 \begin{Lemma} \label{lem3} Let $G_1$ and $G_2$ be finite groups. 
 Assume that a finite group $G$ generates the variety $\V= var\;(G_1,G_2)$. Then
 {\em pot}\;$G = \max(pot\;G_1,pot\;G_2)$. In particular, if $G$ is not monolithic, then
 {\em pot}\;$G=$ {\em pot}\;$(G/N)\le |G/N|$ for a non-trivial normal subgroup $N$ of $G$.
\end{Lemma}
  \proof Clearly, $pot\;G\ge\max (pot\;G_1,pot\;G_2)$. 
 On the other hand, the $n$-generated $\V$-free group $F_n$ is a subdirect product
 of the free groups $F_n'$ and $F_n''$ in the varieties $var\; G_1$ and $var\; G_2$
 (\cite{HN}, 15.82), and hence $|F_n|\le |F'_n||F''_n|$. To estimate $pot\;G$ from above, it remains only to
 note that, for any sequences of positive integers $\{a'_n\}_{n=1}^{\infty}$ and 
 $\{a''_n\}_{n=1}^{\infty}$, we have
 
 $$\limsup_{n\to\infty}\sqrt[n]{\log (a'_na''_n)}=  \limsup_{n\to\infty}\sqrt[n]{\log a'_n+\log a''_n}$$ $$= \max (\limsup_{n\to\infty}\sqrt[n]{\log a'_n}, \; \limsup_{n\to\infty}\sqrt[n]{\log a''_n}).$$
 
 If the group $G$ is not monolithic, there are two nontrivial
 normal subgroups $N_1$ and $N_2$ in $G$ such that $N_1\cap N_2 = \{1\}$. Then $G$
 is a subdirect product of the factor groups $G_1=G/N_1$ and $G_2 = G/N_2$,
 and $var\; G = var\;\{G_1,G_2\}$. Hence $pot\;G = \max(pot\;G_1,pot\;G_2)$,
 which completes the proof.
 \endproof

\begin{Lemma}\label{lem4} Let $Q$ be a group and $N$ be a (right) $Q$-module.
Assume that $N=\oplus_i M_i$ is a direct sum of isomorphic
finite $Q$-submodules $M_i$ of cardinality $m$. Then any $n$-generated
submodule $L$ of $N$ has cardinality at most $m^{nm}$.
\end{Lemma}
\proof There are module isomorphisms $\alpha_{ij}: M_i\to M_j$
such that $\alpha_{jk}\alpha_{ij}=\alpha_{ik}$ for all
$i,j,k$.  If $a\in M_i$ and the indices $i,i_1,...,i_k$
are pairwise different,  then the 'diagonal-type' sum \\
$b=a + \alpha_{ii_1}(a)+\dots+\alpha_{ii_k}(a)$ generates
a $Q$-submodule isomorphic to the submodule generated
by $a$, and so $b$ generates a submodule of order at most $m$.
Since every element $a'$ of $M_j$ equals $\alpha_{ij}(a)$ for some $a\in M_i$, any element of $N$ can be decomposed as a
sum of at most $m$ 'diagonal' elements. Hence every
one-generated submodule of $N$ has order at most $m^m$,
and so $|L|\le (m^m)^n$, as required. \endproof

\begin{Lemma} \label{lem5} Let $G$ be a finite group and $N$ an abelian normal subgroup of a group $H\in var\, G$, on which the finite group $R=H/N$ acts by conjugation,
and $N=\oplus_i M_i$ be a direct sum of isomorphic
finite $R$-submodules $M_i$ of cardinality $m$. Then for every $n\ge 1$, any $n$-generated subgroup $S$ of $H$  has order at most $|G|^{|G|^{\log_2 |R|}}m^{nm}$.
\end{Lemma}
\proof One can choose at most $\log_2|R|$ generators, say $x_1,...,x_k$, in $S$ whose natural images generate the subgroup $Q=SN/N\le R$. By multiplying the other $n-k$ generators by
some products of $x_1^{\pm 1},...,x_k^{\pm 1}$, one obtains a set
of generators $\{x_1,..., x_k, y_{k+1},\dots y_n\}$ in $S$, where $y_i\in N$
for $i>k$.

Since the subgroup $K$ generated by $x_1,...,x_k$ belongs to $var\, G$,
the order of $K$ does not exceed  $C=|G|^{|G|^{\log_2 |R|}}$. The normal
closure $L$ of the elements  $y_{k+1},\dots, y_n$ in $S$ is a $Q$-submodule and so its cardinality is at most $m^{(n-k)m}$ by Lemma \ref{lem4}. Thus,
$|S|=|KL|\le |K||L|\le Cm^{nm}$. \endproof

{\bf Proposition 1.} \label{p1} {\em If $M$ is an abelian normal subgroup of a finite
group $G$, then $pot\;G\le |G/M|$.}

\proof Recall that the $n$-generated ($var\,G$)-free group $F_n$ is embedded
in a direct power $G^I$ for some $I$. Denote by $H$ the product $F_n M^I$,  so $H/M^I$
is a subgroup of $T^I$, where $T=G/M$.
For $i\in I$, let $y_{1i},...,y_{ni}$ be the projections of the free generators $x_1,\dots,x_n$
of $F_n$ in the factor $G_i$, and $z_{1i},...,z_{ni}$ be their natural images in the quotients
$G_i/M_i$. We call two indices $i$ and $j$ equivalent if   the $n$-tuples $(z_{1i},...,z_{ni})$
and $(z_{1j},...,z_{nj})$ are equal (if we identify both $T_i$ and $T_j$ with $T$).
So the set $I$ is a disjoint union of at most $|T|^n$ equivalence classes.

For an equivalence class $J$, we denote by $H_J$ the projection of $H$ to the
direct power $G^J$. It follows from the definition of the equivalence
that the projections of the $n$-tuple of the free generators of $F_n$ to $G_j$ 
are equal modulo $M^J$ for all $j\in J$, whence the image of $F_n$ in $G^J/M^J$
is contained in the diagonal of $(G/M)^J$ and the quotient $R=H_J/M^J$ has order at most $|T|$. 
Moreover, the abelian normal subgroups $M_j$ of $H_J$ are isomorphic $R$-modules
for all $j\in J$.

By Lemma \ref{lem5}, the order of the projection $S_J$ of $F_n$ to $H_J$ does
not exceed  $|G|^{|G|^{\log_2 |T|}}|M|^{n|M|}$. Since the number of
equivalence classes $J$ is at most $|T|^n$, and the group $F_n$ is contained
in the direct product of subgroups $H_J$ over all $J$, we have
$$|F_n|\le   (|G|^{|G|^{\log_2 |T|}}|M|^{n|M|})^{|T|^n}.$$
The logarithm of the right-hand side is bounded from above by
$c n|T|^n $, where the constant $c$ does not depend on $n$.
It follows that $pot\;G\le |T|$, as required.
\endproof

The statement of Theorem \ref{th1} follows from the following

\medskip

{\bf Proposition 2.} \label{p2} {\em (Gap for the values of $pot\;G$). Let $G$ be
a non-trivial finite group. Then $pot\;G=|G|$ if $G$ is a monolithic
group whose monolith is non-abelian. Otherwise $pot\;G\le d<|G|$ for
some divisor $d$ of $|G|$.}

\medskip

\proof The first statement of the proposition is proved in Section \ref{s},
the second one follows from Lemma \ref{lem3} and Proposition 1.
\endproof

\section{Proof of Theorem 2}

Note that a finite prime ring $A$ cannot contain two different minimal ideals,
and multiplication is nonzero on the unique minimal ideal. If, for brevity,
rings with zero multiplication are called abelian (as in Lie theory), $A$
is a monolithic ring with a non-abelian monolith $M$.

Such a reformulation makes Theorem \ref{th2} very similar to Theorem \ref{th1},
and the proofs are similar as well. Of course, the variety $var\; A$ contains
rings, not groups; one works with subrings instead of subgroups and with ideals instead of normal subgroups. If the expression 'centralize each other' is replaced by 'annihilate each other on both sides' in the proof of Lemma 1, we obtain the formulation and proof of an analogous Lemma 1' for rings. 

The obvious changes in vocabulary are also sufficient for the formulation and proof
of a Lemma 2' analogous to Lemma 2. (Of course, one refers to Lemma 1' instead of Lemma 1
in the proof.) The last paragraph of Section \ref{s} can be trivially reformulated as well.
The same changes lead to the formulation and proof of Lemma \ref{lem3}.

Let $R$  be a ring, and $N$ be an abelian ideal of $R$. Then for $a\in N$ and $r\in R$, the products $ra$ and $ar$ depend only on the coset $r+N$, and so the left and right
multiplications of $a$ by the elements of $R/N$ are well defined.
Given two abelian ideals $M$ and $M'$ in $R$, we say that
these ideals are {\it similar in} $R$ if there is an isomorphism of their additive
groups $\alpha: M\to M'$ such that $\alpha(rm) =r\alpha(m)$ and $\alpha(mr)=\alpha(m)r$
for any $m\in M$ and $r\in R$. Clearly, the similarity of ideals is an equivalence
relation.

\medskip

{\bf Lemma 4'.} {\em Let $R$ be a ring, and let $N$ be an ideal of $R$ that is a direct sum
$\oplus_iM_i$ of finite abelian ideals $M_i$ of cardinality $m$ that are similar in $R$.
Then any $n$ elements of $N$ are contained in an ideal $L$ of $R$ of cardinality 
at most $m^{nm}$.}

\medskip

The proof follows that of Lemma \ref{lem4}, where $\alpha_{ij}$ is an
$R$-similarity of the ideals $M_i$ and $M_j$, and the minimal ideals containing
$a\in M_i$ and the 'diagonal' element $b$ (as there) are similar in $R$.

\medskip

{\bf Lemma 5'.} \label{lem5} Let $A$ be a finite ring and $N$ an abelian ideal of a ring $H\in var\, A$, \; $N=\oplus_i M_i$ be a direct sum of similar in $H$ ideals $M_i$ of cardinality $m$, and $R=H/N$. Then for every $n\ge 1$, any $n$-generated subring $S$ of $H$  has cardinality at most $|A|^{|A|^{\log_2 |R|}}m^{nm}$.

\medskip

For the proof of Lemma 5', one changes the vocabulary (in particular, the words 'isomorphic modules' are replaced by 'ideals similar in $G$'). In the proof, the element $y_i$ ($i>k$) is obtained as a sum of $x_i$ and an element from the subring $K$ generated by $x_1,\dots,x_k$. Now $L$ is the minimal ideal of $S$ containing $y_{k+1},\dots,y_n$, and $S=K+L$. 

Only obvious changes are needed in the formulations and proofs of Propositions 1' and 2'.
This completes the proof of Theorem \ref{th2}. The author believes that similar
modifications would lead to the proof of an analogue of Theorem 2 for $\Omega$-rings.

\medskip

{\bf Remark 2.} {\em Let $A$ be a finite-dimensional (non-associative) algebra over 
a finite field $\bf k$, and let $F_n$ denote the $n$-generated free algebra in
the variety $var\,A$ of algebras over $\bf k$ generated by $A$. Then one can
formulate an analogue of Theorem \ref{th2} and prove it without any new assumptions.
In this setting, however, it is preferable to consider the dimension
function $d(n) =\dim_{\bf k} F_n$, which is the logarithm of the cardinality
function $|F_n|$. Thus the potential $pot\;A$ of $A$ is $\limsup_{n\to\infty}\sqrt[n]{d(n)}$,
and Theorem \ref{th2} is modified as follows:

Let $A$ be a finite-dimensional (nonassociative) algebra over a finite
field $\bf k$. The equality $pot\;A=|A|$ holds if and only if $A$ is a prime algebra.

Moreover, if $A$ is a prime algebra, then there exist positive constants $c_1$ and $c_2$ such that\\
$c_1|A|^n <\dim_{\bf k} F_n< c_2 |A|^n$ for the $n$-generated free algebras $F_n$ in the variety $var\;A$.}

\section{Examples and questions}

It is worth finding potentials for various classes of finite groups and rings.
By Theorem \ref{th1} and Lemma \ref{lem3}, it suffices to examine monolithic
groups and rings having an abelian monolith.

\medskip

{\bf Example 1.} Let $G$ be a finite dihedral group. If it is a 2-group, then $pot\;G=1$,
because $G$ is nilpotent. Otherwise $pot\;G\ge 2$ as mentioned in the Introduction.
On the other hand, $G$ has an abelian subgroup of index $2$, and therefore 
$pot\;G=2$ in this case by Proposition 1.

\medskip

{\bf Example 2.} Let $G$ be a Miller--Moreno group, that is a non-abelian finite group,
all of whose proper subgroups are abelian. We assume that $G$ is not nilpotent,
since otherwise $pot\;G=1$. Then, for two distinct prime numbers $p$ and $q$,
$G=PQ$, where $P$ is an elementary abelian $p$-group, $Q$ is a cyclic $q$-group,
and $Q^q$ lies in the center of $G$ (see \cite{R}, 9.1.9). So any non-abelian monolithic 
quotient $\bar G$ of $G$ belongs to the variety $\V$ consisting of groups having
a normal elementary abelian $p$-subgroup with an elementary abelian $q$-group as a quotient. Then $\bar G$ generates the variety $\V$ (see \cite{HN}, 54.41).
The order of the free group $F_n$ in $\V$ is equal to $q^np^{(n-1)q^n+1}$ (\cite{HN},
21.13), whence $pot\;G=pot\;\bar G =q.$  For example, $pot\;A_4=3$ for the 
alternating group $A_4$ of degree $4$.

\medskip

{\bf Example 3.} The following product of varieties of associative algebras (no $1$
in the signature) over a field was introduced in the paper of G. Bergman and J. Lewin 
\cite{BL}. The product $\U\V$ of two varieties is defined by the set of all
identities $u(x_1,\dots,x_k)v(x_{k+1},\dots x_{k+\ell}) = 0$, where $u(x_1,\dots,x_k)=0$
is an identity of the variety $\U$ and $v(x_{1},\dots x_{\ell}) = 0$ is
an identity of $\V$. Assume that $\U=var\;A$ and $\V=var\;B$, where $A$ and $B$
are non-zero finite associative algebras over a finite field $\bf k$. Then, by
 Proposition 18 and Theorem 20 of \cite{BO}, the variety $\U\V$ is generated by
 a finite algebra $C$ over $\bf k$ such that $pot\;C =(pot\;A)(pot\;B)$.
 
 \medskip

{\bf Example 4.} Given a finite field $\bf k$, there is only one non-abelian $2$-dimensional
Lie algebra $L$ over $\bf k$ up to isomorphism. It has basis $(e,f)$ such 
that $ef= -fe =f$ and $e^2=f^2 =0$. There is a 1-dimensional monolith $\la f\ra$, and
$pot\;L \le |L/\la f\ra|=|\bf k|$ by Proposition 1'. To obtain the opposite inequality, we construct
an $n$-generated algebra $A$ in $var\; L$ as follows.

Consider the set $V^*$ of $|{\bf k}|^{n-1}$ linear functions $\alpha: V\to \bf k$, where $V$ is an 
$n-1$-dimensional abelian Lie algebra over $\bf k$. For each $\alpha\in V^*$, let $\la f_{\alpha}\ra$ be a $1$-dimensional Lie algebra and define the Lie module $M_{\alpha}=\la f_{\alpha}\ra$
over $V$ by the rule $vf_{\alpha}=\alpha(v)f_{\alpha}$ for $v\in V$. These rules define
the semidirect sum $A$ of $V$ and the ideal $I=\oplus_{\alpha\in V^*} M_{\alpha}$.
The ideal $I$ is generated by one element $\sum_{\alpha}f_{\alpha}$ since the weights
$\alpha$ are pairwise distinct. So $A$ is an $n$-generated Lie algebra of
dimension $n+|{\bf k}|^{n-1}$. It is embedded in a direct power of $L$ whence $A\in var\;L$.
(For each $\alpha$, consider a homomorphism that sends $v$ to $\alpha(v)e$, sends 
$f_{\alpha}$ to $f$, and kills the other weight spaces.)
Since $\lim_{n\to\infty}\sqrt[n]{n+|{\bf k}|^{n-1}} = |{\bf k}|$, we have
$pot\;L\ge |{\bf k}|$.

\medskip 

{\bf Question 1.} Is the potential $pot\;G$ an integer for every finite group $G$?

\medskip

{\bf Question 2.} Does $\lim_{n\to\infty}\sqrt[n]{\log |F_n|}$ exist for
any finite group $G$, where $F_n$ are the $n$-generated free groups in the variety $var\; G$?

\medskip

{\bf Question 3.} In the notation of Question 2, does 
$\lim_{n\to\infty}\frac{\log |F_{n+1}|}{\log |F_n|}$ exist? (Then it is equal
to the limit from Question 2 and to $pot\;G$.)

\medskip

Similar questions for finite algebras have been formulated in \cite{BO}.

If a finite group $G$ has a central subgroup $Z$ and the quotient $G/Z$ is
simple (in particular, if $G$ is quasisimple), then $pot\;G = pot\;(G/Z)$
by Theorem \ref{th1} and Proposition 1. The equality $pot\; G = pot\;(G/Z)=1$ also
holds if $G/Z$ is a non-trivial nilpotent group. So I venture to ask

\medskip

{\bf Question 4.} Let $Z$ be a proper subgroup (subring) of a finite group (ring) $G$
that centralizes (resp., annihilates) $G$. Does the equality $pot\;G = pot\;(G/Z)$ hold?

\medskip

If a group $G$ has an abelian normal subgroup $N$ and $Q = G/N$, there is
a semidirect product $\bar G =QN$, $Q\cap N= \{1\}$, such that the action
by conjugation of $Q$ on $N$ coincides in $G$ and in $\bar G$. The
following construction (depending on $N$) shows that $\bar G =\bar G(N)$ belongs to $var\; G$. 

Let $H$ be the direct product $G_1\times G_2$ of two copies of $G$ and $D$
be the diagonal of $G$. By definition,  $K= DN_1 = DN_2$, where $N_1$ and $N_2$
are copies of $N$. Note that $D\cap N_1=D\cap N_2 = \{1\}$ and $L= D\cap (N_1\times N_2)$
is the diagonal of $N_1\times N_2$. It is a normal subgroup of $K$ since $N$ is abelian.
Denote the quotient $K/L$ by $\bar G$. Then the image $Q$ of $D$ in $K/L =\bar G$ has trivial 
intersection with the image of $N_1$ (= image of $N_2$) in $\bar G$, and $\bar G$
is the required semidirect product.

The same construction defines the semidirect sum $\bar G\in var\; G$ if $G$ is a ring and $N$ is
an abelian ideal of $G$. This leads to the following generalization of Question 4.

\medskip 

{\bf Question 5.} Let $N$ be an abelian normal subgroup (ideal) in a finite
group (resp., ring) $G$. In the above notation, does the equality $pot\;G = pot\;(\bar G)$ hold? 

\medskip

{\bf Acknowledgments}. The author thanks Professor Yu. Bahturin for useful discussions
and the anonymous reviewer for many valuable comments that have helped to improve and
to shorten the presentation.

\end{document}